\documentclass[final]{siamltex}
\usepackage[sumlimits]{amsmath}
\usepackage{amsfonts}
\usepackage{xspace}
\usepackage{graphicx,graphics,color}
\DeclareMathOperator{\diff}{d}
\DeclareMathOperator{\Ro}{Ro}
\DeclareMathOperator{\Fr}{Fr}
\newcommand{\MM}[1]{{\boldsymbol #1}}

\newcommand{\dd}[2]{\frac{\diff#1}{\diff#2}}

\newcommand{\rem}[1]{}
\newtheorem{remark}[theorem]{Remark}
\newtheorem{example}[theorem]{Example}
\newcommand{\pdgp}{$\textrm{P1}_{\textrm{DG}}$-$\textrm{P2}$\xspace}
\newcommand{\pndgpn}{$\textrm{Pn}_{\textrm{DG}}$-$\textrm{P(n+1)}$\xspace}
\begin{document}

\author{Colin Cotter\thanks{Department of Aeronautics, Imperial 
    College London, Prince Consort Road, London SW7 2AZ,
        UK ({\tt colin.cotter@imperial.ac.uk}).}}

\title{A family of mixed finite element pairs with optimal geostrophic 
balance}

\maketitle

\begin{abstract}
We introduce a family of mixed finite element pairs for use
  on geodesic grids and with adaptive mesh refinement for numerical
  weather prediction and ocean modelling. We prove that when these
  finite element pairs are applied to the linear rotating shallow
  water equations, the geostrophically balanced states are exactly
  steady, which means that the numerical schemes do not introduce any
  spurious inertia-gravity waves; this makes these finite element
  pairs in some sense optimal for numerical weather prediction and
  ocean modelling applications. We further prove that these finite
  element pairs satisfy an inf-sup condition which means that they are
  free of spurious pressure modes which would pollute the numerical
  solution over the timescales required for large-scale geophysical
  applications. We then discuss the extension to incompressible
  Euler-Boussinesq equations with rotation, and show that for the
  linearised equations the balanced states are again exactly steady on
  arbitrary unstructured meshes. We also show that the discrete
  pressure Poisson equation resulting from these discretisations
  satisfies an optimal stencil property. All these properties make the
  discretisations in this family excellent candidates for numerical
  weather prediction and large-scale ocean modelling applications when
  unstructured grids are required.
\end{abstract}




\pagestyle{myheadings}
\thispagestyle{plain}
\markboth{C. COTTER}{Mixed elements with optimal geostrophic balance}

\section{Introduction}

The aim of this paper is to introduce a new family of finite element
pairs for discretising large-scale ocean and atmosphere flows, state
theorems about their wave propagation properties, and to explain why
these properties are important for numerical weather prediction and
ocean modelling. Section \ref{background} provides some motivational
background, section \ref{model} describes the balance properties of
model equations that we aim to preserve exactly in the numerical
discretisations, and section \ref{geostrophic states} explains how
these properties can be modified by numerical discretisations. Section
\ref{family} introduces the family of finite element pairs, and the
optimal balance property is proved in section \ref{properties}. It is
also shown, by proving an inf-sup condition, that the finite element
pairs are free from spurious pressure modes. In section \ref{3d} these
properties are extended to the case of three-dimensional rotating
stratified incompressible flow which applies to non-hydrostatic ocean
modelling. Finally, section \ref{summary} provides a summary and
outlook to future developments.

\subsection{Background}
\label{background}
An operational weather forecasting system combines observational data
(such as satellite images, and pressure measurements on the Earth's
surface, for example) with a numerical model of the atmosphere (solved
on state-of-the-art parallel computers) to produce weather
forecasts. The numerical model consists of a mathematically-consistent
discretisation of the equations of motion for the atmosphere (known as
the dynamic core) together with schemes (known as parameterisations)
for representing subgrid scale physics such as convection and
turbulence, and physical processes such as cloud formation and
atmospheric chemistry.  Most current state-of-the art numerical
weather prediction (NWP) models, such as the MET Office Unified Model
\cite{Da+2005}, make use of a latitude-longitude grid to construct the
numerical approximations to the fluid dynamical equations that form
the \emph{dynamical core} of the model. However, recently there has
been growing interest in more general horizontal discretisation
schemes constructed using triangles or hexagons. This is for two
reasons. Firstly, geodesic grids (which are obtained by iterative
refinement of an icosahedron using triangles or the dual grid which
comprises a combination of hexagons plus exactly 12 pentagons) provide
a much more uniform coverage over the sphere, which has possible
advantages for accurate representation of wave propagation and avoids
very fine grid cells near to the North and South poles. A number of
groups are now using geodesic grids for weather and climate models
\cite{Ri+2000,Ma+2002,Sa+2008}. Secondly, triangles facilitate
adaptive mesh refinement much more easily, allowing regional models to
be nested seamlessly in a global model, and even allowing dynamic mesh
refinement in which the mesh is dynamically adapted during a
forecast. However, the introduction of adaptively-refined triangular
grids calls for the careful design of new numerical schemes which
correctly represent the large scale geophysical balances so that model
forecasts can be run over sufficiently long times (the current
forecast window is about a week). In this paper we introduce a new
family of numerical discretisation methods on triangular grids which
shall be shown to optimally represent these geophysical balances on
arbitrary unstructured grids. We also extend these properties to three
dimensional rotating stratified incompressible flow so that they may
be applied to non-hydrostatic ocean modelling.

\subsection{Model equations and geostrophic states}
\label{model}
As a model problem, we consider the
shallow-water equations on an $f$-plane
\begin{eqnarray}
\label{swe 1}
\MM{u}_t + (\MM{u}\cdot\nabla)\MM{u} + f\MM{u}^\perp
+ g\nabla D &=& 0, \quad \MM{u} = (u_1,u_2), \quad \MM{u}^\perp = (-u_2,u_1), \\
\quad D_t + \nabla\cdot(\MM{u}D) &=& 0, \label{swe 2}
\end{eqnarray}
where $\MM{u}$ is the horizontal velocity, $D=\bar{D}+\eta$ is the
total layer depth, $\eta$ is the perturbation layer thickness,
$\bar{D}$ is the mean layer thickness, $\MM{k}$ is the unit vector in
the $z$-direction, $f$ is the (constant) Coriolis parameter and $g$ is
the acceleration due to gravity. The boundary conditions are
\begin{equation}
\label{bcs}
\MM{u}\cdot\MM{n} = 0 \quad\mathrm{on}\quad \partial\Omega,
\end{equation}
where $\partial\Omega$ denotes the boundary of the two-dimensional
domain $\Omega$, and $\MM{n}$ is the normal to $\partial\Omega$. These
equations model the dynamics of a layer of hydrostatic incompressible
fluid with constant density with a free surface and columnar motion so
that the horizontal velocity is independent of depth. This can be
thought of as a simple model for the ocean, or for a single layer in
the atmosphere. For a derivation of these equations, see
\cite{Sa1998}, for example. Numerical methods for NWP and ocean
modelling are often developed in two dimensions using the rotating
shallow-water equations. The methods can then be extended to the
\emph{primitive equations} (three-dimensional equations of rotating
stratified hydrostatic flow, see \cite{Sa1998}) by building a mesh in
layers in which the two-dimensional discretisation is applied in the
horizontal plane. Hence, the shallow-water equations provide a useful
testbed for numerical schemes and a stepping stone to the primitive
equations which already poses many of the key challenges for designing
good numerical methods for NWP and ocean
modelling.

A key dimensionless number in geophysical fluid dynamics is the Rossby
number, defined by
\[
\Ro = U/fL,
\]
where $U$ is a typical velocity scale, and $L$ is a typical horizontal
length scale; the Rossby number measures the relative importance of
the acceleration and Coriolis terms. Geophysical flow problems in the
small Rossby number limit are concerned with the states which satisfy
\[
f\MM{u}^\perp 
+ g\nabla D \approx 0
\]
so that the Coriolis term approximately balances the pressure
gradient. This is the state of \emph{geostrophic balance}. It has long
been observed (see \cite{McNo2000,MoDr2001}, for example) that if the
initial conditions for the state variables $\MM{u}_0(\MM{x})$ and
$D_0(\MM{x})$ are initialised near to geostrophic balance, then this
state will be approximately preserved for very long times. This has
observed to be the case both in the low Rossby number limit, and also
in the $\mathcal{O}(1)$ limit and various mechanisms have been
proposed for the maintenance and breakdown of this balance in various
parameter regimes \cite{FoMcNo2000,CoRe2006,OlOlVa2008}). Large scale
flows in the atmosphere and ocean (such as those associated with the
global circulation that determines global weather and climate) are
observed to be in a state of geostrophic balance and hence it is
important that numerical schemes used for weather forecasting, ocean
modelling and climate prediction can correctly represent these
balances.

The ability of a numerical scheme to represent geostrophic balance can
be examined by studying wave propagation properties. Choosing a flat
topography so that $\bar{D}$ is a constant, the linearised shallow-water
equations are
\begin{eqnarray}
\label{lswe 1}
\MM{u}_t + f\MM{u}^\perp 
+ g\nabla \eta &=& 0, \\
\quad \eta_t + \bar{D}\nabla\cdot\MM{u} &=& 0. \label{lswe 2}
\end{eqnarray}
If the velocity $\MM{u}$ and free surface elevation $\eta$ are chosen
in a state of perfect geostrophic balance so that
\[
f\MM{u}^\perp + g\nabla\eta=0,
\]
then the velocity divergence satisfies
\[
\nabla\cdot\MM{u} = \frac{g}{f}\nabla\cdot\nabla^\perp\eta
= -\eta_{yx}+\eta_{xy} = 0, \qquad \nabla^\perp = (-\partial_y,\partial_x)
\]
and so we have a steady state
\[
\MM{u}_t = 0, \quad \eta_t = 0.
\]
Hence, all geostrophically-balanced states are steady for the
linearised equations. In the nonlinear case, the solutions variables
evolve through the nonlinear advection terms which are small in the
low Rossby number limit, and hence the geostrophic balance is
maintained as a form of adiabatic invariance due to the small
parameter $\Ro$. 

If periodic boundary conditions are chosen, then one can perform a dispersion
analysis for solutions of equations (\ref{lswe 1}-\ref{lswe 2}) by 
substituting the \emph{ansatz}
\[
\MM{u}=\hat{\MM{u}}e^{i\left(\MM{k}\cdot\MM{x}-\omega t\right)},
\qquad \eta=\hat{\eta}e^{i\left(\MM{k}\cdot\MM{x}-\omega t\right)},
\]
resulting in the dispersion relation
\begin{equation}
\label{dispersion}
\omega\left(\omega^2 - f^2 - g\bar{D}\left(k^2+l^2\right)\right)=0.
\end{equation}
This equation has three roots, with the $\omega=0$ root corresponding
to the steady geostrophic state; this branch becomes the \emph{Rossby
  wave} branch when the model is extended to the $\beta$-\emph{plane
  model} in which the Coriolis parameter $f$ varies in the
$y$-direction so that $f=f_0+\beta y$ for constants $f_0$ and
$\beta$. The other two roots give rise to dispersive waves, known as
\emph{inertia-gravity waves}; the $k=l=0$ case is known as an
\emph{inertial oscillation} in which the fluid undergoes solid body
motion with a flat free surface.  Since all of the roots are real, the
state of geostrophic balance is stable under small perturbations.

\subsection{Geostrophic states for numerical methods}
\label{geostrophic states}
It is crucial that numerical methods for equations (\ref{swe
  1}-\ref{swe 2}) do not generate spurious inertia-gravity waves when
the solution is near to geostrophic balance. This can typically occur
if the wave-propagation properties of the numerical method
(\emph{i.e.} the numerical discretisation of the linearised equations)
do not correctly represent this balance.  Given a discretisation of
the linear system (\ref{lswe 1}-\ref{lswe 2}) it is simple to check
the evolution of geostrophically-balanced states under this
discretisation. One constructs initial conditions which satisfy the
discrete form of geostrophic balance, steps the variables forward in
time, and inspects the variables to check that the steady state is
approximately preserved. This analysis has been performed for various
element pairs in \cite{RoStLi1998}. In general, the discrete
divergence of the velocity field will not be exactly zero, and the
remainder due to numerical discretisation errors will lead to
oscillations. Whether the numerical method is suitable for computing
the evolution of geostrophically-balanced states over long time
intervals (\emph{i.e.}  suitable for weather forecasting or ocean
modelling) depends on how these errors behave. If one computes the
numerical dispersion relation (\emph{i.e.} the numerical analogue of
equation (\ref{dispersion}) for the method (see for example the
calculations in \cite{Th08,Ro+2005}) then it is possible to divide the
eigenmodes of the system into geostrophic modes which converge to the
geostrophic states as the mesh edge-lengths converge to zero, and
inertia-gravity modes which converge to the inertia-gravity waves. If
the numerical discretisation errors in the divergence of the balanced
states are large and project onto the inertia-gravity modes, then
large unbalanced dynamics will be apparent after a long time
integration interval (such as the time interval that is relevant to
weather forecasting). When solving the nonlinear equations these
numerical errors are constantly generated by the nonlinear terms,
resulting in the geostrophic component of the solution being polluted
by spurious inertia-gravity waves which render the numerical scheme
useless for NWP and global ocean modelling.

In section \ref{family} we  present a family of mixed
finite element pairs which have the optimal property that the
geostrophically balanced states are exactly steady; this means that
these numerical discretisations are in some sense optimal for
geophysical fluid dynamics problems. This property is independent of
the choice of mesh, which can be taken to be completely
unstructured. This means that the numerical discretisations can be
used to solve geophysical flow problems in the presence of mesh
refinement and adaptivity. This is proved in theorem \ref{steady thm}
in section \ref{properties}. We shall also show that the finite
element pairs satisfy an inf-sup condition which means that they are
free from spurious pressure modes: eigenmodes which have very small
discrete gradients of free surface elevation despite having a free
surface which is not flat. The absence of these modes is also crucial
for geophysical applications since they can be coupled to the physical
modes through the nonlinear terms and eventually becoming as large as
the physical solution; the existence of these modes prohibits the use
of such numerical methods in weather forecasting and ocean modelling. This
result is proved in theorem \ref{inf-sup thm}, also in section
\ref{properties}. The proofs of these results are very simple and
elegant, due to the geometric embedding conditions that define the
family of discretisations.

\section{Family of finite element pairs}
\label{family}

In this section we introduce our family of finite element pairs, first
by developing the general finite element formulation in section
\ref{fem formulation}, and then by stating conditions which define our
particular family in section \ref{choice}.
\subsection{Finite element formulation}
\label{fem formulation}
In this subsection we develop the finite element approximation to the
linearised shallow-water equations by writing down the weak form of
the equations and restricting the function spaces to chosen spaces of
piecewise polynomials on a finite element mesh.  We start with the
linearised shallow-water equations on an $f$-plane given in equations
(\ref{lswe 1}-\ref{lswe 2}) with boundary conditions given by equation
(\ref{bcs}).  To obtain the weak form of the equations we multiply
equation (\ref{lswe 1}) by a test function $\MM{w}$ and equation
(\ref{lswe 2}) by a test function $\phi$ and integrate over the domain
$\Omega$ to obtain
\begin{eqnarray}
\label{weak 1}
\dd{}{t}\int_{\Omega} \MM{w}\cdot\MM{u}\diff{V} 
+ f\int_{\Omega} \MM{w}\cdot\MM{u}^\perp\diff{V}
& = &
-g\int_{\Omega}\MM{w}\cdot\nabla \eta\diff{V}, \\
\dd{}{t}\int_{\Omega} \phi \eta \diff{V} & = & 
-\bar{D}\int_{\Omega} \phi\nabla\cdot\MM{u} \diff{V}.
\label{weak 2}
\end{eqnarray}
We then integrate equation (\ref{weak 2}) by parts, and make use of
the boundary conditions (\ref{bcs}) to obtain 
\begin{eqnarray}
\label{weak parts 1}
\dd{}{t}\int_{\Omega} \MM{w}\cdot\MM{u}\diff{V} 
+ f\int_{\Omega} \MM{w}\cdot\MM{u}^\perp\diff{V}
& = &
-g
\int_{\Omega}\MM{w}\cdot\nabla \eta\diff{V}, \\
\label{weak parts 2}
\dd{}{t}\int_{\Omega} \phi \eta \diff{V} & = & 
\bar{D}\int_{\Omega} \nabla\phi\cdot\MM{u} \diff{V},
\end{eqnarray}
which must hold for all test functions $\MM{w}\in H^1(\Omega)$ and $\phi\in
L_2(\Omega)$. This is the weak form of equations (\ref{lswe 1}-\ref{lswe 2})
which we shall discretise using the Galerkin finite element method.

The Galerkin projection of equations (\ref{weak parts 1}-\ref{weak
  parts 2}) is constructed by defining finite dimensional spaces for
the numerical solution variables $\MM{u}^\delta$ and $\eta^\delta$,
and the test functions $\MM{w}^\delta$ and $\phi^\delta$. We shall use
a mixed finite element method (see \cite{AuBrLo2004} for an excellent
general survey), which means that one type of finite element space
shall be used for $\MM{u}^\delta$ and $\MM{w}^\delta$, and a different
type of finite element space shall be used for $\eta^\delta$ and
$\phi^\delta$.

We shall begin by defining the possible finite element spaces in
general terms, before going on to state conditions which define our
particular family of discretisations.
\begin{definition}[Finite element mesh]
  Let the mesh $\mathcal{M}$ be a set of non-overlapping polygons
  (elements) which completely cover the computational domain $\Omega$
  which has a elementwise polygonal boundary $\partial\Omega$.
\label{fem}
\end{definition}
\begin{definition}[Pressure space]
  Let $H$ be a space of elementwise polynomials on $\mathcal{M}$, of
  type and continuity to be specified. 
\end{definition}
This is a general definition, but we note that we shall require at
least $C^0$ continuity across element boundaries, since we apply
gradients to $\phi^\delta$ and $\eta^\delta$ in equations
(\ref{galerkin 1}-\ref{galerkin 2}).
\begin{definition}[Velocity space]
  Let $V$ be a space of vectors of elementwise polynomials on
  $\mathcal{M}$, of type and continuity to be specified (possibly
  differently to $H$).
\label{vel space}
\end{definition}
We note that we do not require any continuity conditions for $V$, since
gradients are not applied to $\MM{u}^\delta$ and $\MM{w}^\delta$ in equations
(\ref{galerkin 1}-\ref{galerkin 2}).

Having defined $H$ and $V$, we may now write down the Galerkin finite
element method for equations (\ref{lswe 1}-\ref{lswe 2}), which is
obtained by restricting the solution variables and the test functions
to these finite dimensional spaces.
\begin{definition}[Galerkin finite element method]
  $\MM{u}^\delta(\MM{x},t)$ and $\eta^\delta(\MM{x},t)$ are the
  semi-discrete solutions of the Galerkin finite element
  discretisation of (\ref{lswe 1}-\ref{lswe 2}) if
\[
\MM{u}^\delta(\cdot,t)\in V, \quad \eta^\delta(\cdot,t) \in H, \qquad
\forall t \in [0,T],
\]
and
\begin{eqnarray}
\label{galerkin 1}
\dd{}{t}\int_{\Omega} \MM{w}^\delta\cdot\MM{u}^\delta\diff{V} 
+ f\int_{\Omega} \MM{w}^\delta\cdot{\MM{u}^\delta}^\perp\diff{V}
& = &
-g
\int_{\Omega}\MM{w}^\delta\cdot\nabla \eta^\delta\diff{V}, \\
\label{galerkin 2}
\dd{}{t}\int_{\Omega} \phi^\delta \eta^\delta \diff{V} & = & 
\bar{D}\int_{\Omega} \nabla\phi^\delta\cdot\MM{u}^\delta \diff{V},
\end{eqnarray}
for all test functions $\MM{w}^\delta\in V$, $\phi^{\delta}\in H$.
\end{definition}
This equations may be solved on a computer by expanding
$\MM{w}^\delta$ and $\MM{u}^\delta$ in a basis for $V$, and
$\phi^\delta$ and $h^\delta$ in a basis for $H$, which produces a
matrix equation for the basis coefficients of $\MM{u}^\delta$ and
$h^\delta$. This equation may then be discretised in time using a
suitable time integration method. 

\subsection{Choice of finite element spaces}
\label{choice}
In defining the problem, it remains to select a particular choice
spaces $(V,H)$ (known as a finite element pair). In this paper, we
discuss a large family of possible choices defined by the following
condition:
\begin{definition}[Embedding conditions]
\label{embedding}
\begin{enumerate}
\item The operator $\nabla$ defined by the pointwise gradient
\[
{q}^\delta(\MM{x}) = \nabla h^\delta(\MM{x})
\]
maps from $H$ into $V$.
\item The skew operator $\perp$ defined by the pointwise
formula
\[
\MM{q}^\delta(\MM{x}) = (\MM{u}^\delta(\MM{x}))^\perp
\]
maps from $V$ into itself.
\end{enumerate}
\end{definition}
These conditions are most definitely not satisfied by all possible
pairs $(V,H)$, as illustrated by the following examples.
\begin{example}[P1-P1]
  The finite element pair known as P1-P1 (which may be used for the
  shallow-water equations but requires stabilisation as described in
  \cite{WaCa1998}) is defined as follows:
\begin{itemize}
\item The mesh $\mathcal{M}$ is composed of triangular elements.
\item $H$ is the space of elementwise-linear functions $h^\delta$
  which are continuous across element boundaries.
\item $V$ is the space of vector fields $\MM{u}^\delta$ with both of
  the Cartesian components $(u^\delta,v^\delta)$ in $H$.
\end{itemize} 
Condition 1 of Definition \ref{embedding} is not satisfied by the
P1-P1 pair since gradients of functions in $H$ are discontinuous
across element boundaries. Condition 2 is satisfied since the same
continuity conditions are required for normal and tangential
components.
\end{example}
\begin{example}[RT0]
  The lowest order Raviart-Thomas \cite{RaTh1977} velocity space
  (known as RT0) is constructed on a mesh $\mathcal{M}$ composed of
  triangular elements. It consists of elementwise constant vector
  fields which are constrained to have continuous normal components
  across element boundaries. RT0 does not satisfy condition 2 of
  Definition \ref{embedding} since the $\perp$ operator transforms
  vector fields with discontinuities in the tangential component
  (which are permitted in RT0) into vector fields with discontinuities
  in the normal component (which are not).
\end{example}

We now describe some examples of finite element pairs which \emph{do}
satisfy the conditions in Definition \ref{embedding}.
\begin{example}[P0-P1]
  The finite element pair known as P0-P1 (applied to ocean modelling
  in \cite{Um+2004}, for example) is defined as follows:
\begin{itemize}
\item The mesh $\mathcal{M}$ is composed of triangular elements.
\item $H$ is the space of elementwise-linear functions $h^\delta$
  which are continuous across element boundaries.
\item $V$ is the space of elementwise-constant vectors with discontinuities
  across element boundaries permitted. 
\end{itemize}
\end{example}
\begin{example}[\pdgp]
  The finite element pair known as \pdgp \cite{CoHaPaRe2009} is
  defined as follows:
\begin{itemize}
\item The mesh $\mathcal{M}$ is composed of triangular elements.
\item $H$ is the space of elementwise-quadratic functions $h^\delta$
  which are continuous across element boundaries.
\item $V$ is the space of elementwise-linear vectors with discontinuities
  across element boundaries permitted. 
\end{itemize}
\end{example}
Each of these examples satisfy both conditions in Definition
\ref{embedding}: condition 1 holds because taking the gradient of a
elementwise polynomial $n-1$ which is continuous across element
boundaries results in a vector field which is discontinuous across
element boundaries and is composed of elementwise polynomials of one
degree $n$, and condition 2 holds since the velocity space uses the
same continuity constraints for normal and tangential components
\emph{e.g.}  both components are allowed to be discontinuous. This
defines a whole sequence of high-order \pndgpn element pairs. Similar
elements can be constructed on quadrilateral elements.  Since we only
require these two conditions to prove our optimal balance property
which holds on arbitrary meshes, we can also construct finite element
spaces on mixed meshes composed of quadrilaterals and triangles, for
example. It is also possible to use $p$-adaptivity in which different
orders of polynomials are used in different elements, as long as the 
conditions are satisfied.

\section{Geostrophic balance properties}
\label{properties}
In this section we prove that when finite element pairs which are
chosen to satisfy both of the conditions in Definition
\ref{embedding}, their discrete geostrophically-balanced velocities
satisfy the discrete divergence-free condition exactly: this means
that the discrete geostrophically-balanced states are exactly steady
states of equations (\ref{galerkin 1}-\ref{galerkin 2}). This is the
result of Theorem \ref{steady thm} in this section, which makes use of
conditions 1 and 2 of Definition \ref{embedding}. We also prove an
inf-sup condition for these finite element pairs (making use of
condition 1 of Definition \ref{embedding}) which provides a lower
bound (independent of edge-lengths in the mesh) for the discrete
gradient operator applied to non-constant functions in $H$. This lower
bound prohibits the existence of spurious pressure modes which render
a finite element pair unsuitable for geophysical flow problems; it
also allows one to prove the convergence of the numerical solutions at
the optimal rate obtained from approximation theory. This condition is
stated and proved in Theorem \ref{inf-sup thm}.

\subsection{Optimal geostrophic balance}

We first prove the following lemma which illustrates the embedding 
properties of our family of finite element pairs.
\begin{lemma}[Pointwise gradient lemma]
\label{pointwise gradient lemma}
  Let $(H,V)$ be a finite element pair chosen to satisfy condition 1
  in Definition \ref{embedding}. Let Let $\MM{q}^\delta\in V$ by the
  discrete gradient of $\eta^\delta\in H$ defined by
\begin{equation}
\label{discrete gradient}
\int_{\Omega} \MM{w}^\delta\cdot\MM{q}^\delta\diff{V}
 = 
\int_{\Omega}\MM{w}^\delta\cdot\nabla \eta^\delta\diff{V},
\end{equation}
for all test functions $\MM{w}^\delta\in V$. Then $\MM{q}^\delta$ is
the pointwise (strong) gradient of $\eta^\delta$ defined by
\[
\MM{q}^\delta(\MM{x}) = \nabla \eta^\delta(\MM{x}), \quad \forall
\MM{x}\in \Omega.
\]
\end{lemma}
\begin{proof}
Since condition 1 is satisfied, we may choose a test function
\[
\MM{w}^\delta = \MM{q}^\delta-\nabla\eta^\delta\in V,
\]
and substitution into equation \ref{discrete gradient} gives
\begin{eqnarray*}
0&=&\int_{\Omega}
\MM{w}^\delta\cdot\left(\MM{q}^\delta-\nabla\eta^\delta\right)\diff{V} \\
&=&\int_{\Omega} \left|\MM{q}^\delta-\nabla\eta^\delta\right|^2\diff{V} \\
&=& \|\MM{q}^\delta -\nabla\eta^\delta\|^2_{L_2}.
\end{eqnarray*}
Since the $L_2$-norm only vanishes for elements of $H$ if they are
identically zero, we conclude that $\MM{q}^\delta=\nabla\eta^\delta$
as required.
\end{proof}
This lemma appears at first sight to be a tautology but since the
discrete gradient $\MM{q}^\delta$ can be thought of as the
$L_2$-projection of $\nabla\eta^\delta$ into $V$, it requires
condition 1 of Definition \ref{embedding} to be satisfied, and it is
not the case for equal-order element pairs such as P1-P1, for
example. We shall make use of this lemma in proving Theorem
\ref{inf-sup thm}. 

The following lemma extends this technique to show that if the
discrete geostrophic balance relation is satisfied by functions taken
from a finite element pair in our family of discretisations, then the
exact geostrophic balance condition is actually satisfied at each
point.
\begin{lemma}[Embedding lemma]
  \label{embedding lemma}
  Let $(H,V)$ be a finite element pair chosen to satisfy both
  conditions in Definition \ref{embedding}. 
Let $\MM{u}^\delta\in V$
  and $\eta^\delta\in H$ satisfy the discrete geostrophic balance
  relation
\begin{equation}
\label{discrete balance}
f\int_{\Omega} \MM{w}^\delta\cdot{\MM{u}^\delta}^\perp\diff{V}
 = 
-g
\int_{\Omega}\MM{w}^\delta\cdot\nabla \eta^\delta\diff{V},
\end{equation}
for all test functions $\MM{w}^\delta\in V$. Then 
\begin{equation}
\label{strong balance}
f(\MM{u}^\delta)^\perp(\MM{x}) = 
-g\nabla \eta^\delta(\MM{x}), \quad \forall
\MM{x}\in \Omega.
\end{equation}
\end{lemma}
\begin{proof}
Following the technique of the previous lemma, we note that
conditions 1 and 2 mean that we may choose a test function
\[
\MM{w}^\delta = f{\MM{u}^\delta}^\perp + g\nabla\eta^\delta\in V,
\]
and substitute into equation \ref{discrete balance} to
obtain
\begin{eqnarray*}
0&=&\int_{\Omega}
\MM{w}^\delta\cdot\left(f{\MM{u}^\delta}^\perp+g\nabla\eta^\delta
\right)\diff{V} \\
&=&\int_{\Omega} \left|f{\MM{u}^\delta}^\perp+g\nabla\eta^\delta
\right|^2\diff{V} \\
&=&\|f{\MM{u}^\delta}^\perp +g\nabla\eta^\delta\|^2_{L_2},
\end{eqnarray*}
hence the result.
\end{proof}

This lemma is useful for initialising the system variables in a
balanced state since the geostrophic balance relation can be evaluated
pointwise instead of requiring a projection to be computed.  Next we
apply this lemma to prove the following theorem, which states that
finite element pairs have the optimal property that these
geostrophically balanced states are steady solutions of equations
(\ref{galerkin 1}-\ref{galerkin 2}).
\begin{theorem}[Steady geostrophic states]
\label{steady thm}
Let $\MM{u}^\delta\in V$ and $\eta^\delta\in H$ satisfy equation
(\ref{discrete balance}), and let $\partial\Omega$ be a contour
for $\eta^\delta$ (so that the balanced velocity field obtained from
equation \ref{strong balance} satisfies the boundary condition).
Then $\MM{u}^\delta$ and $\eta^\delta$ are
steady solutions of equation (\ref{galerkin 1}-\ref{galerkin 2}).
\end{theorem}
\begin{proof}
  Substitution into equation (\ref{galerkin 1}), and choosing
$\MM{w}^\delta=\MM{u}^\delta$ gives
\[
\dd{}{t}\int_{\Omega} |\MM{u}^\delta|^2\diff{V}=0, 
\]
and hence $\MM{u}_t^\delta=0$. It remains to show that
\begin{equation}
\label{steady eta}
\dd{}{t}\int_\Omega\phi^\delta\eta^\delta\diff{V} =
\bar{D}\underbrace{\int_\Omega\nabla\phi^\delta\cdot\MM{u}^\delta\diff{V}}
_{\mbox{divergence integral}} = 0,
\end{equation}
for all test functions $\phi^\delta\in H$. 

By lemma \ref{embedding lemma}, equation \ref{strong balance} is 
satisfied, which we may substitute into the divergence integral
to obtain 
\begin{eqnarray}
\int_\Omega\nabla\phi^\delta\cdot\MM{u}^\delta\diff{V} & = &
\frac{g}{f}\int_\Omega \nabla\phi^\delta\cdot\nabla^\perp\eta^\delta\diff{V}.
\end{eqnarray}
The right-hand side integral in this equation can be shown to vanish
for all $\phi^\delta$ and $\eta^\delta$ in $H^1$ (which contains our
velocity space $H$): the proof is obtained by taking $\phi^\delta$ and
$\eta^\delta$ as the limit of a convergent sequence of continuous
functions in $H^1$ for which the sequence can be shown to vanish after
integration by parts (see \cite{Ev1998}, for example). Here we provide
a more direct proof which is obtained by integrating by parts
separately in each element (since the gradients are discontinuous
across element boundaries) to obtain
\begin{eqnarray*}
\int_\Omega \nabla\phi^\delta\cdot\nabla^\perp\eta^\delta\diff{V} 
&=& \sum_{E\in\mathcal{M}}
\int_E \nabla\phi^\delta\cdot\nabla^\perp\eta^\delta\diff{V}  \\
\{\mbox{integration by parts}\}
& = & -\sum_{E\in\mathcal{M}}\int_E\phi^\delta
\underbrace{\nabla\cdot\nabla^\perp\eta^\delta}_{=0}\diff{V} \\
 & & \quad +\sum_{E\in\mathcal{M}}\int_{\partial E}
\phi^\delta\MM{n}\cdot\nabla^\perp\eta^\delta\diff{V} \\
 & = & \sum_{\Gamma\in\mathcal{M}, \Gamma\cap\partial\Omega=\emptyset}\int_{\Gamma}
\underbrace{[[\phi^\delta\nabla^\perp\eta^\delta]]}_{=0}\diff{V} \\
& & \quad + \int_{\partial\Omega}\phi^\delta\underbrace{\MM{n}\cdot\nabla^\perp
\eta^\delta}_{=0}\diff{S},
\end{eqnarray*}
where $E$ indicates an element in the mesh $\mathcal{M}$, $\partial E$
is the boundary of $E$ with outward pointing normal $\MM{n}$, $\Gamma$
indicates an edge in the mesh $\mathcal{M}$ with an arbitrary chosen
normal direction, and $[[\MM{u}]]$ indicates the jump in the normal
component of a vector field $\MM{u}$ across an edge $\Gamma$. In the
final line, the first term vanishes since $\phi^\delta$ is continuous
across element boundaries, and the tangential components of $\nabla
\phi^\delta$ are also continuous (to see this, note that the
tangential gradient of $\phi^\delta$ on the boundary is obtained by
integrating in the direction of the boundary where $\eta^\delta$ is
continuously-differentiable). The second term vanishes since the
domain boundary $\partial\Omega$ is a contour for $\eta^\delta$ (from
the assumption in the theorem), and so the tangential derivative of
$\eta^\delta$ vanishes there.

Hence, equation (\ref{steady eta}) is satisfied. Choosing 
$\phi^\delta=\eta^\delta$ yields
\[
\dd{}{t}\int_\Omega\left(\eta^\delta\right)^2\diff{V} = 0,
\]
and hence $\eta^\delta_t=0$ and we have a steady state solution of
equations (\ref{galerkin 1}-\ref{galerkin 2}).
\end{proof}
\begin{remark}
This proof corrects the proof presented in \cite{CoHaPa2009}.
\end{remark}
\begin{remark}
Note that this theorem does not depend in any way on the structure of
the mesh $\mathcal{M}$ and so applies to arbitrary finite element
discretisations on unstructured meshes, provided that the conditions
of Definition \ref{embedding} are satisfied. 
\end{remark}

We checked this theorem numerically by taking a completely
unstructured mesh, randomly generating $\eta^\delta$ and
$\MM{u}^\delta$ fields which satisfy the conditions of the theorem,
and integrated equations (\ref{galerkin 1}-\ref{galerkin 2}) using the
implicit-midpoint rule. The problem was solved in dimensionless
variables in a $1\times 1$ square domain with Rossby number $\Ro=0.1$
and Froude number $\Fr=1.0$, with a timestep size $\Delta t=0.01$ for
1000 steps. The maximum relative error between the initial and final
$\eta^\delta$ fields was numerically zero (\emph{i.e.} round-off error
was observed) for each random realisation over hundreds of tests. Some
example fields are shown in figure \ref{randoms}. These images
illustrate that the optimal balance result is completely independent
of the mesh and the smoothness of the solution. Some more general
convergence tests using Kelvin waves are provided in
\cite{CoHaPa2009}.

\begin{figure}
\begin{center}
\hspace{0mm}
\includegraphics*[width=6cm]{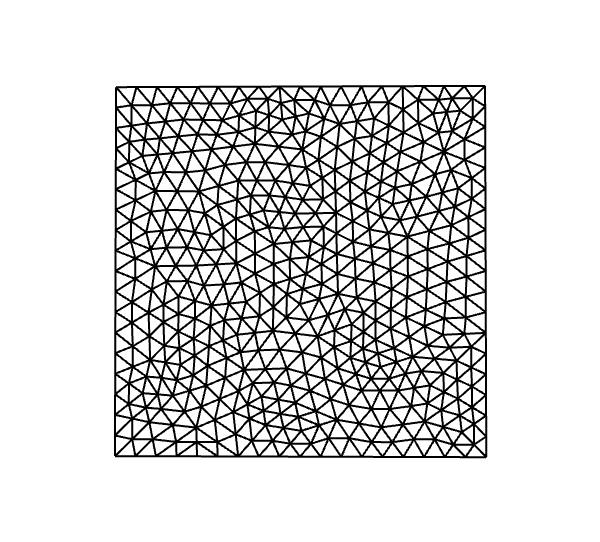}
\includegraphics*[width=6cm]{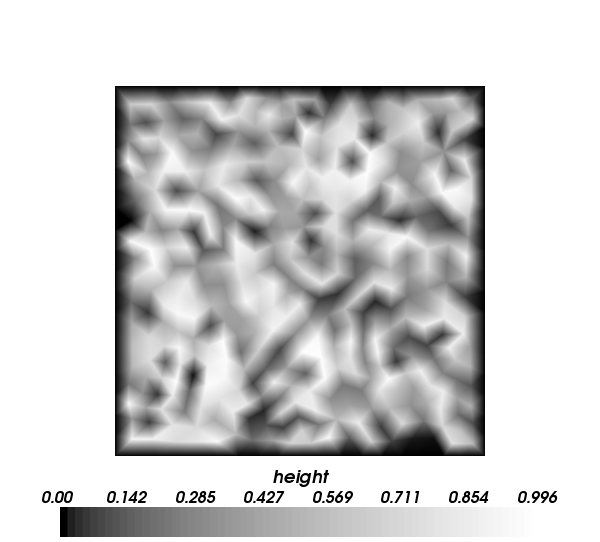} \\
\includegraphics*[width=6cm]{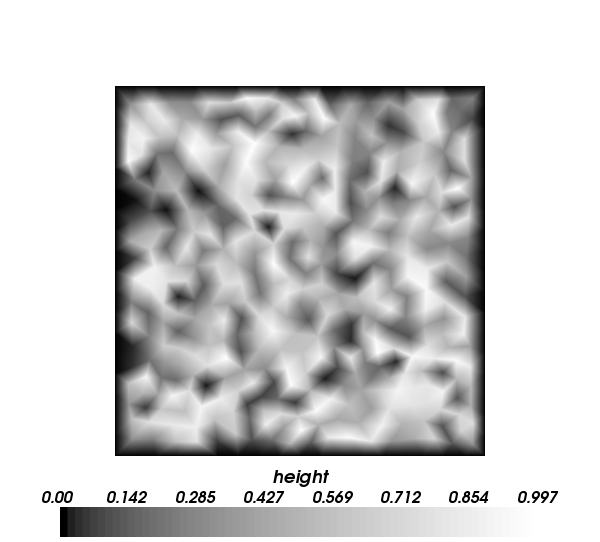}
\includegraphics*[width=6cm]{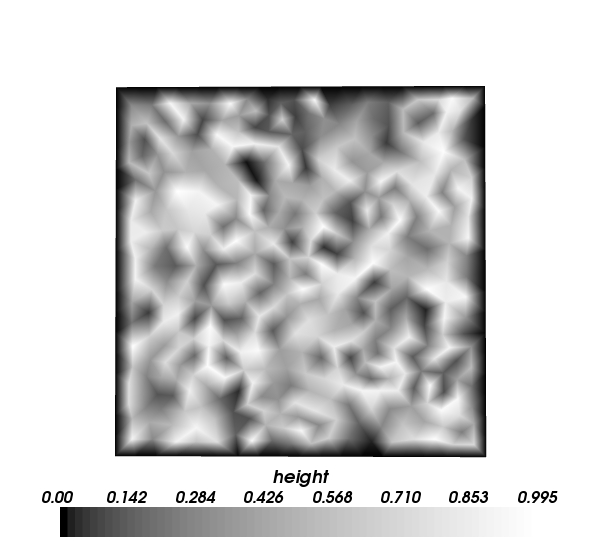} \\
\end{center}
\caption{\label{randoms} Figure showing random $\eta^\delta$ fields
  used to verify the optimal balance property described in Theorem
  \ref{steady thm}. Randomly generated $\eta^\delta$ and
  $\MM{u}^\delta$ fields (constrained to satisfy the balance
  conditions of the theorem) were used as initial conditions for a
  numerical integration with the implicit-midpoint rule in time
  applied to equations (\ref{galerkin 1}-\ref{galerkin 2}). The
  difference between the initial and final $\eta^\delta$ fields was
  observed to be numerically zero for all random realisations that
  were tested. Note that the steady state does not depend on the
  smoothness of the solution or on the mesh structure.}
\end{figure}

\subsection{Inf-sup condition: pressure modes}
We next prove that the discretisation given by
equations (\ref{galerkin 1}-\ref{galerkin 2}) using a finite element
pair satisfying our two conditions does not have any spurious pressure
modes. These are modes which converge (in the limit as the mesh
edge-lengths go to zero) to steady solutions with zero velocity even
though the free surface displacement $\eta^\delta$ is not flat. These
modes are catastrophic for discretisations of the nonlinear equations
since they can grow to dominate the solution when coupled to the
physical modes. These modes are not present if it is possible to bound
the gradient of non-constant functions in $H$ from below as the
edge-lengths go to zero. This bound is expressed in the following
inf-sup condition:
\begin{definition}[Inf-sup condition]
The inf-sup condition for a finite element pair $(H,V)$ requires that
\begin{equation}
\label{inf-sup eqn}
\sup_{\MM{w}^\delta\in V}\frac{\int_\Omega\MM{w}^\delta
\cdot\nabla\phi^\delta\diff{V}}
{\sqrt{\int_{\Omega}|\MM{w}^\delta|^2\diff{V}}}
\geq \beta \sqrt{\int_{\Omega}(\phi^\delta)^2\diff{V}},
\quad \forall \phi^\delta \in H,
\end{equation}
where $\beta$ is independent of the edge lengths in the mesh
$\mathcal{M}$.
\label{inf sup condition}
\end{definition}
\begin{theorem}[Inf-sup theorem]
\label{inf-sup thm}
Let $(H,V)$ be a finite element pair which satisfies condition 1 in
definition \ref{embedding}. Then $(H,V)$ satisfies the inf-sup condition
in definition \ref{inf sup condition}. 
\end{theorem}
\begin{proof}
  The supremum in equation (\ref{inf-sup eqn}) is bounded from below
  by any particular choice of test function $\MM{w}^\delta$. By
  condition 1 $\nabla\phi^\delta\in V$ and so we may choose
  $\MM{w}^\delta=\nabla\phi^\delta$. Substitution then gives
\begin{eqnarray*}
\sup_{\MM{w}^\delta\in V}\frac{\int_\Omega\MM{w}^\delta
\cdot\nabla\phi^\delta\diff{V}}
{\sqrt{\int_{\Omega}|\MM{w}^\delta|^2\diff{V}}} & 
\geq & \frac{\int_\Omega\nabla\phi^\delta
\cdot\nabla\phi^\delta\diff{V}}
{\sqrt{\int_{\Omega}|\nabla\phi^\delta|^2\diff{V}}} \\
& = & \sqrt{\int_{\Omega}|\nabla\phi^\delta|^2\diff{V}} \\
& = & \sqrt{\frac{\int_{\Omega}|\nabla{\phi^\delta}|^2\diff{V}}
{\int_{\Omega}{\phi^\delta}^2\diff{V}}}
\sqrt{\int_{\Omega}{\phi^\delta}^2\diff{V}}, \\
& \geq &\sqrt{\lambda_{\min}}\sqrt{\int_{\Omega}{\phi^\delta}^2\diff{V}},
\end{eqnarray*}
where $\lambda_{\min}$ is the minimum non-zero eigenvalue of 
(minus) the Laplacian on $\Omega$, having made use of the Rayleigh
quotient
\[
\lambda_{\min} = \min_{\phi^\delta\neq 0}
\frac{\int_{\Omega}|\nabla\phi^\delta|^2\diff{V}}
{\int_{\Omega}{\phi^\delta}^2\diff{V}}.
\]
\end{proof}

We note that the family of finite element pairs considered in this
paper corresponds to one particular option for discrete differential
forms satisfying a discrete de Rham complex condition described in
\cite{ArFaWi2006}. In fact, condition 1 of Definition \ref{embedding}
corresponds to all of the options described in that paper.

A consequence of the inf-sup theorem, as discussed in
\cite{AuBrLo2004,ArFaWi2006b}, is that the solutions of the wave
equation converge at the optimal rate described by the theory of
numerical interpolation, as described in the following corollary:
\begin{corollary}
  Given an interval $[0,T]$, there exists a constant $C(T)$ such that
\[
\|\MM{u}(\cdot,T)-\MM{u}^\delta(\cdot,T)\|_{L_2} + \|\eta(\cdot,T)-\eta^\delta(\cdot,T)\|_{L_2} \leq
C(T)h^{k+1}
\]
where $(\MM{u},\eta)$ is the solution of equations (\ref{weak
  1}-\ref{weak 2}) with initial conditions $(\MM{u}_0,\eta_0)$,
$(\MM{u}^\delta,\eta^\delta)$ is the solution of equations
(\ref{galerkin 1}-\ref{galerkin 2}) with initial conditions
$(\MM{u}^\delta_0,\eta^\delta_0)$ which satisfy the interpolation
condition
\[
\|\MM{u}^\delta_0-\MM{u}_0\|_{L_2} +
\|\eta^\delta_0-\eta_0\|_{L_2}\leq ch^{k+1}
\]
and where $k$ is the minimum of the orders of the elementwise
polynomials used to construct $\MM{u}^\delta$ and $\eta^\delta$.
\end{corollary}
\begin{proof}
  The proof follows using standard mixed finite element techniques,
  namely obtaining a bound on the $L_2$-norm of the solution variables
  $\MM{u}^\delta$ and $\eta^\delta$ which requires the inf-sup
  condition. See \cite{Jo2003}, \cite{AuBrLo2004} or
  \cite{ArFaWi2006b}, for example.
\end{proof}

For the \pdgp element pair, this convergence property (in this case
2nd-order convergence since $k=1$) was confirmed from numerical
experiments for the 2-dimensional wave equation in
\cite{CoHaPaRe2009}, and for the rotating linear shallow-water
equations on an $f$-plane in \cite{CoHaPa2009}.

\section{Three dimensional incompressible flow}
\label{3d}
In this section we briefly discuss the extension of these properties
to the equations of three dimensional rotating stratified
nonhydrostatic incompressible flow (Boussinesq equations) which,
together with their hydrostatic counterparts, are used in ocean
modelling.

\subsection{Model equations and geostrophic states}

The full equations of motion are:
\begin{eqnarray}
  \MM{u}_t + (\MM{u}\cdot\nabla)\MM{u} + f\MM{k}\times\MM{u} 
  & = & -\frac{1}{\rho_0}\nabla p + \MM{k}b, \label{momentum} \\
  \nabla\cdot\MM{u} & = & 0, \label{continuity} \\
  T_t + \MM{u}\cdot\nabla T & = & 0, \label{temp}
\end{eqnarray}
where $\MM{u}$ is the three-dimensional velocity field, $f$ is the
Coriolis parameter, $\MM{k}$ is the unit upward vector, $p$ is the
pressure, $b$ is the buoyancy and $T$ is the temperature (salinity
would also be included in a full ocean model but does not add anything
to this discussion). The system is closed by specifying an equation of
state in which the buoyancy $b$ is defined as a function of $T$ and
$p$. Here we have made the \emph{traditional approximation} which
restricts the rotation vector to the vertical axis, and the $f$-plane
approximation for which $f$ is a constant.

In three dimensions there are two geophysical balances which are present
in slowly-varying large-scale flows which arise when the acceleration
terms are small compared to the Coriolis and buoyancy. In the vertical 
direction we obtain the hydrostatic balance
\begin{equation}
\label{hydros}
-\frac{1}{\rho_0}p_z + b = 0,
\end{equation}
which can be imposed as a constraint in a hydrostatic model, allowing
the pressure to be computed explicitly from the buoyancy by vertical
integration. In the horizontal direction we again obtain the
geostrophic balance
\begin{equation}
u = -\frac{1}{\rho_0f}p_y, \quad v = \frac{1}{\rho_0f}p_x.
\label{geobal}
\end{equation}
It is simple to check again that $\nabla\cdot\MM{u}=0$ for these
balanced states. In the case of incompressible flow, this means that 
the balanced states are solutions of the equations given above.

One can again take a numerical discretisation of equations
(\ref{momentum}-\ref{temp}), construct the discrete balanced solutions
and check if the discrete form of equation \ref{continuity} is
satisfied exactly. If there is some residual, then this means that it
is not possible for exactly-balanced states to exist in the numerical
discretisation, which can lead to to the generation of spurious
internal inertia-gravity waves. For example, if a pressure-projection
method is used for timestepping (as is typical for non-hydrostatic
models) then the time-integrator has two stages: the first stage takes
a momentum step using the pressure from the previous timestep, then
the solution is projected back to satisfy the discrete form of
equation \ref{continuity}. If the balanced states do not satisfy this
equation, then each projection will generate further spurious
unbalanced motion.

In this section we shall briefly describe an extension of our family
of finite-element pairs to the three-dimensional case, and describe
the extension of our optimal balance results. Since we are still
concerned with wave propagation we linearise equations 
(\ref{momentum}-\ref{temp}) about the state
\[
\MM{u}=0, \quad p_z = \rho_0b, \quad T = \bar{T}(\MM{z}),
\]
to obtain 
\begin{eqnarray}
\label{linear momentum}
  \MM{u}_t + f\MM{k}\times\MM{u} & = & 
-\frac{1}{\rho_0}\nabla p' + \MM{k}\gamma T', \\
\label{linear div u}
  \nabla\cdot\MM{u} & = & 0, \\
  T'_t + u_3\bar{T}_z & = & 0, \label{linear temp}
\end{eqnarray}
where $\gamma$ is a suitable positive constant. We shall drop the
primes for the rest of the section. For these equations, steady
balanced states given by
\[
u_3 = 0, \quad p_z = \rho_0\gamma T, \quad u = -\frac{1}{\rho_0f}p_y,
\quad v = \frac{1}{\rho_0f}p_x,
\]
satisfy $\nabla\cdot\MM{u}=0$ and hence are admissible solutions of
the equations.

\subsection{Finite element formulation}

Defining a Galerkin finite element method for these equations requires
us to define a finite element space $\Theta$ for the temperature
variable $T^\delta$. The choice of this space is not very important
for the discussion of geostrophically-balanced states, so it will not
be developed much here, except to note that it \emph{is} important to
ensure that there are sufficiently many states which satisfy a
discrete hydrostatic balance, otherwise the representation of
hydrostatic balance will be poor, leading to spurious non-hydrostatic
motion. The velocity space $V$ and the pressure space $H$ are defined
by the three dimensional extensions of definitions \ref{fem}-\ref{vel
  space}. 

We now define the Galerkin finite element method for equations
(\ref{linear momentum}-\ref{linear temp}) as follows:
\begin{definition}[Galerkin finite element method for 3D wave propagation]
  $\MM{u}^\delta(\MM{x},t)$, $p^\delta(\MM{x},t)$, $T^\delta(\MM{x},t)$ are the
  semi-discrete solutions of the Galerkin finite element
  discretisation of (\ref{linear momentum}-\ref{linear temp}) if
\[
\MM{u}^\delta(\cdot,t)\in V, \quad p^\delta(\cdot,t) \in H, \quad
T^\delta(\cdot,t)\in \Theta, \qquad \forall t \in [0,T],
\]
and
\begin{eqnarray}
\label{galerkin 3d 1}
\dd{}{t}\int_{\Omega} \MM{w}^\delta\cdot\MM{u}^\delta\diff{V} 
+ f\int_{\Omega} \MM{w}^\delta\cdot\MM{k}\times\MM{u}^\delta\diff{V}
& = &
-\frac{1}{\rho_0}
\int_{\Omega}\MM{w}^\delta\cdot\nabla p^\delta\diff{V}
+ \gamma\MM{k}\cdot\int_\Omega\MM{w}^\delta T^\delta \diff{V}
, \\
\label{galerkin 3d 2}
-\int_{\Omega} \nabla\phi^\delta\cdot\MM{u}^\delta \diff{V} & = & 0, \\
\dd{}{t}\int_{\Omega}\theta^\delta T^\delta \diff{V} + \MM{k}\cdot\int_{\Omega}
\theta^\delta\MM{u}^\delta \bar{T}^\delta_z\diff{V} & = & 0, 
\label{galerkin 3d 3}
\end{eqnarray}
for all test functions $\MM{w}^\delta\in V$, $\phi^{\delta}\in H$, 
$\theta^\delta\in \Theta$.
\end{definition}
We again require our element pair $(H,V)$ to satisfy the conditions in
Definition \ref{embedding}, extended to three dimensions (with the
$\perp$ operator replaced by the $\MM{k}\times$ operator). We next
define the discrete geophysical balances using this discretisation.

\begin{definition}[Discrete hydrostatic and geostrophic balance]
The solution variables $\MM{u}^\delta$, $p^\delta$ and $T^\delta$
satisfy the hydrostatic and geostrophic balances if
\begin{equation}
\label{hydro geo discrete}
f\int_{\Omega} \MM{w}^\delta\cdot\MM{k}\times\MM{u}^\delta\diff{V}
 = 
-\frac{1}{\rho_0}
\int_{\Omega}\MM{w}^\delta\cdot\nabla p^\delta\diff{V}
+ \gamma\MM{k}\cdot\int_\Omega\MM{w}^\delta T^\delta \diff{V}.
\end{equation}
The vertical component specifies hydrostatic balance and the 
horizontal component specifies geostrophic balance.
\end{definition}
This definition allows us to state the optimal balance theorem for
three-dimensional incompressible flow, which we give in the next
subsection.
\subsection{Optimal balance properties, inf-sup theorem and optimal pressure
matrix property}

\begin{theorem}[Optimal balance for three-dimensional incompressible flow]
  Let $\MM{u}^\delta\in V$, $p^\delta \in H$ and $T^\delta \in \Theta$
  be chosen so that equation (\ref{hydro geo discrete}) is satisfied,
  with zero vertical velocity $u^\delta_3=0$ and the pressure
  $p^\delta$ satisfying the pointwise condition
\[
\MM{\tau}_H\cdot\nabla p^\delta = 0, \quad
\forall \MM{x}\in \partial\Omega,
\]
where $\tau_H$ is the horizontal tangent vector to $\partial\Omega$ so
that the boundary is a streamline for the balanced flow (consistent
with the boundary condition for $\MM{u}^\delta)$.  Then
$(\MM{u}^\delta,p^\delta,T^\delta)$ is a steady state solution of
equations (\ref{galerkin 3d 1}-\ref{galerkin 3d 3}).
\end{theorem}
\begin{proof}
  It is simple to check that $\MM{T}^\delta_t=0$, $\MM{u}^\delta_t=0$
  by inserting the solutions into the equations and noting that all of
  the terms vanish, and it remains to check that equation
  \ref{galerkin 3d 2} is satisfied. The proof proceeds exactly as the
  proof of Theorem \ref{steady thm}, by first noting that the two
  conditions in Definition \ref{embedding} imply that 
\[
\MM{u}^\delta = (-\psi^\delta_y,\psi^\delta_x), \quad \psi^\delta =
\frac{1}{f\rho_0}p^\delta
\]
pointwise, which we then insert into equation \ref{galerkin 3d 2}
to obtain
\[
-\int_{\Omega} \nabla\phi^\delta\cdot\MM{u}^\delta \diff{V}  =  
\int_{\Omega} \phi^\delta_x\psi^\delta_y - \phi^\delta_y\psi^\delta_x \diff{V}  =  
0,
\]
using a similar argument to the previous proof.
\end{proof}

This means that balanced states are exactly steady and do not generate
any spurious internal waves. Note that this theorem is again
completely independent of the mesh structure and so the finite element
pairs in this family are ideal for representing balanced flows on
unstructured meshes such as those proposed in \cite{Pa+2005}.

As for the 2D case, it is necessary for the finite element spaces to
satisfy an inf-sup condition, so that solutions of the linearised
equations converge at the optimal rate defined by approximation
theory. In the case of incompressible flow, one also forms a pressure
Poisson equation by composing the discrete divergence and gradient
operators, and if the inf-sup condition is not satisfied then there
are very small spurious eigenvalues in the matrix which make iterative
solvers very slow (see \cite{GrSa2000}, for example). Our family of
finite element pairs satisfy the inf-sup condition in three dimensions
which can be shown by simple extension of the proof of Theorem
\ref{inf-sup thm}. For incompressible flow, we shall use the
same techniques to prove further properties of the discrete Poisson
matrix. 

For the continuous equations, the Poisson equation is formed by taking
the divergence of equation (\ref{linear momentum}) and applying
equation (\ref{linear div u}) to obtain
\begin{equation}
\label{pressure poisson}
\nabla^2 p = \nabla\cdot\MM{r} = \nabla\cdot(
\rho_0\MM{k}\gamma T -   \rho_0f\MM{k}\times\MM{u}),
\end{equation}
which specifies an equation for $p$ given $T$ and $\MM{u}$. This
equation must be solved at each timestep to calculate the pressure
field. When the Galerkin finite element method is applied, this
specifies a coupled system of equations for $\MM{p}$ given by
\begin{eqnarray}
\label{p eqn 1}
\int_{\Omega}\MM{w}^\delta\cdot\MM{r}^\delta\diff{V}
& = &
- \rho_0f\int_{\Omega} \MM{w}^\delta\cdot\MM{k}\times\MM{u}^\delta\diff{V}
+ \rho_0\gamma\MM{k}\cdot\int_\Omega\MM{w}^\delta T^\delta \diff{V}
, 
\\
\label{p eqn 2}
\int_{\Omega}\MM{w}^\delta\cdot\MM{q}^\delta\diff{V} 
&=&\int_{\Omega}\MM{w}^\delta\cdot\nabla\MM{p}^\delta\diff{V}, \\
\label{p eqn 3}
\int_{\Omega} \nabla\phi^\delta\cdot\MM{q}^\delta \diff{V} & = & 
-\int_{\Omega} \nabla\phi^\delta\cdot\MM{r}^\delta\diff{V},
\end{eqnarray}
for $p^\delta\in H$, $\MM{q}^\delta,\MM{r}^\delta\in V$ and for all
test functions $\phi^\delta\in H$ and $\MM{w}^\delta \in V$. In practise,
the variables $\MM{q}^\delta$ and $\MM{r}^\delta$ are eliminated to
obtain an equation for $p^\delta$. This then ensures that equation
(\ref{galerkin 3d 2}) is satisfied at each instance in time (or each
timestep, having discretised the equations in time). This system gives
rise to a matrix equation for the basis function coefficients which,
in general can have a larger sparsity pattern since it involve the
product of several matrices (see \cite{GrSa2000}, for example).  In the
following theorem, we show that this sparsity pattern is reduced when
condition 1 of Definition \ref{embedding} is satisfied, which is a
further useful property of the family of finite element pairs
discussed in this paper.

\begin{theorem}[Optimal sparsity of pressure matrix]
Let $p^\delta$, $\MM{q}^\delta$, $\MM{r}^\delta$ be the solutions of
equations (\ref{p eqn 1}-\ref{p eqn 2}). Then
\begin{equation}
\label{galerkin}
\int_{\Omega} \nabla\phi^\delta\cdot\nabla p^\delta \diff{V}  = 
-\int_{\Omega} \nabla\phi^\delta\cdot\MM{r}^\delta\diff{V},
\end{equation}
for all test functions $\phi^\delta\in H$, which is the usual Galerkin
discretisation for equation (\ref{pressure poisson}).
\end{theorem}
\begin{proof}
We again make use of lemma \ref{pointwise gradient lemma} which states
that 
\[
\MM{q}^\delta = \nabla p^\delta
\]
at each point. Substitution into equation (\ref{p eqn 3}) gives the result.
\end{proof}
Since the usual Galerkin discretisation of equation (\ref{pressure
  poisson}) results in a single equation that does not require the
elimination of variables, this results in a much sparser stencil for
the matrix equation for the basis function coefficients of $p^\delta$.

\section{Summary and outlook}
\label{summary}
In this paper we defined a large family of finite element
discretisations for the rotating shallow-water equations and the three
dimensional equations of rotating stratified incompressible flow. When
applied to the linear rotating shallow water equations, these
discretisations were shown satisfy the optimal property that
geostrophically-balanced states are completely steady, which mirrors a
property of the solutions of the continuous equations. It was also
shown that the discretisations in the family satisfy an inf-sup
condition which prohibits the existence of spurious pressure
modes. This makes the discretisations in the family strong
candidates for use in NWP and ocean modelling in cases where
triangular elements are required, \emph{e.g.} to allow the use of
geodesic grids. Furthermore, the proofs are independent of the mesh
structure which means that the family of discretisations produce
stable results in the presence of adaptive mesh refinement.  We then
discussed the extension of the family to three-dimensional
incompressible flow, required for ocean modelling, and showed that the
discretisations in the family result in exactly steady balanced states
on completely arbitrary unstructured meshes in three dimensions.  In
addition, the properties of the family were used to show that the
matrix obtained from the discretised pressure Poisson equation has an
optimally sparse stencil.

In future work, we will develop and test discretisations of the
fully-nonlinear equations using choices from our family of finite
element spaces, particularly from the \pndgpn sub-family applied to
the rotating shallow-water equations and the three dimensional
incompressible equations. Elements of that sub-family have
discontinuous velocity, which allows a discontinuous Galerkin
treatment of the advection terms in the momentum equation. We also
plan to compute numerical dispersion relations for these
discretisations when applied to the geodesic grid, in order to make
comparisons with other element pairs and finite difference methods on
this grid. The \pdgp element is currently being implemented in the
Imperial College Ocean Model, and will be developed into a
$p$-adaptive scheme in which different order polynomials are used in
different elements.

\bibliography{fox1}

\end{document}